\documentclass{article}

\usepackage{indentfirst}
\usepackage{amsmath}
\usepackage{amsthm}
\usepackage{amsfonts}
\usepackage{tikz-cd}
\usepackage{url}

\theoremstyle{plain}
\newtheorem{thm}{Theorem}[section]
\newtheorem{prop}[thm]{Proposition}
\newtheorem{lem}[thm]{Lemma}
\newtheorem{conj}[thm]{Conjecture}

\newtheorem{rmk}[thm]{Remark}
\newtheorem{prob}[thm]{Problem}

\DeclareMathOperator{\codim}{codim}
\DeclareMathOperator{\can}{can}

\title{On Miyanishi conjecture for quasi-projective varieties}
\author{Takumi Asano}
\date{}

\begin{document}

\maketitle

\begin{abstract}
Miyanishi conjecture claims that for any variety over an algebraically closed field of characteristic zero, any endomorphism of such a variety which is injective outside a closed subset of codimension at least $2$ is bijective. We prove Miyanishi conjecture for any quasi-projective variety $X$ which is a dense open subset of a $\mathbb{Q}$-factorial normal projective variety $\overline{X}$ such that $\codim (\overline{X} \setminus X) \ge 2$ with the ample canonical divisor or the ample anti-canonical divisor. Also, we observe Miyanishi conjecture without the conditions of its canonical divisor by using minimal model program. In particular, we prove Miyanishi conjecture in the case that $\overline{X}$ has canonical singularities and $\overline{X}$ has the canonical model which is obtained by divisorial contractions.
\end{abstract}

\section{Introduction}
We work over an algebraically closed field $k$ of characteristic zero, and varieties over $k$ mean separated integral schemes of finite type over $k$. We call $1$-dimensional varieties curves.

J. Ax proved the following theorem in \cite{Ax}.

\begin{thm}[{\cite[Theorem]{Ax}}]\label{A}
Let $A$ and $B$ be schemes such that $A$ is of finite type over $B$. For any endomorphism $\varphi$ of $A$ over $B$, if $\varphi$ is injective, then $\varphi$ is bijective.
\end{thm}

The following generalization of Theorem \ref{A} is conjectured by M. Miyanishi in \cite{Open}.

\begin{conj}[Miyanishi conjecture, \cite{Open}]\label{M}
Let $\varphi:X \to X$ be an endomorphism of a variety $X$ over $k$, and let $Y$ be a closed subset of $X$ which is of codimension at least $2$. If $\varphi$ is injective on $X \setminus Y$, then $\varphi$ is an automorphism.
\end{conj}

S. Kaliman showed that Conjecture \ref{M} is true in the case that $X$ is affine or complete over $k$ in \cite{Kal}, and N. Das showed that Conjecture \ref{M} is true in other several cases, especially when $X$ is smooth, in \cite{Das}.

The following two theorems are our main results which give new conditions on which Conjecture \ref{M} is true.

\begin{thm}\label{main1}
Let X be a dense open subset of a $\mathbb{Q}$-factorial normal projective variety $\overline{X}$ over $k$ with $\codim(\overline{X} \setminus X) \ge 2$. If either the canonical divisor $K_{\overline{X}}$ of $\overline{X}$ or $-K_{\overline{X}}$ is ample, then any endomorphism $\varphi$ of $X$ which satisfies the conditions of Conjecture \ref{M} is an automorphism.
\end{thm}

\begin{thm}\label{main2}
Let $X$ be a dense open subset of a $\mathbb{Q}$-factorial normal projective variety $\overline{X}$ over $k$ and $\codim(\overline{X} \setminus X) \ge 2$. Suppose $\overline{X}$ has canonical singularities and $\overline{X}$ has the canonical model $\overline{X}^{\can}$ which is obtained by minimal model program. If each of the steps of this minimal model program is a divisorial contraction, then Conjecture \ref{M} holds for $X$.
\end{thm}

In minimal model program, birational maps which are not divisorial contractions often appear, called flips. The following is also one of our results which tells us that if Conjecture \ref{M} does not hold for threefold $X$ as in Theorem \ref{main2}, then flips which appear in minimal model program have to satisfy some conditions.

\begin{thm}\label{main3}
Let $X$ be a dense open subset of a $\mathbb{Q}$-factorial normal projective threefold $\overline{X}$ over $k$ and $\codim(\overline{X} \setminus X) \ge 2$. Suppose $\overline{X}$ has canonical singularities and $\overline{X}$ has the canonical model $\overline{X}^{\can}$ which is obtained by minimal model program. Let $\overline{X}=\overline{X}_0 \dashrightarrow \overline{X}_1 \dashrightarrow \cdots \dashrightarrow \overline{X}_n=\overline{X}^{\can}$ be a sequence given by minimal model program to obtain $\overline{X}^{\can}$ from $\overline{X}$. Let $\varphi$ be an endomorphism of $X$ which satisfies the conditions of Conjecture \ref{M} but is not an automorphism. Then there is a curve $C$ in $X$ such that $\varphi(C)$ is a point $P \in X$. In the way described in Section $4$, $\varphi_0=\varphi$ induces $\varphi_i:\overline{X}_i \dashrightarrow \overline{X}_i$ for all $i$, and $\overline{C}_0=\overline{C}$, which is the closure of $C$ in $\overline{X}$, transforms into a curve $\overline{C}_i$ in $\overline{X}_i$ for all $i$, and $P_0=P$ transforms into a point $P_i \in \overline{X}_i$ for some $i$. Then, there exists some $i$ such that $f_i:\overline{X}_i \dashrightarrow \overline{X}_{i+1}$ is a flip and neither of the following conditions holds.

\begin{enumerate}
\item $P_i$ is defined and $f_i$ is defined at $P_i$.\\

\item There exists an open subset $U$ of $\overline{X}_i$ such that $f_i$ induces an isomorphism between $U$ and $f_i(U)$, and $U$ contains $\overline{C}_i$.
\end{enumerate}
\end{thm}

In the rest of this section, we introduce some useful results which are already known. 

S. Kaliman used the following two useful lemmas in \cite{Kal}.

\begin{lem}[{\cite[Lemma 2]{Kal}}]\label{normal}
If Conjecture \ref{M} is true for normal varieties over $k$, then it is true for all varieties $X$. 
\end{lem}

\begin{lem}[{\cite[Lemma 3]{Kal}}]\label{Z}
Let $Z=X \setminus \varphi(X \setminus Y)$. Then, $Z$ is a closed subset of $X$ and $\dim Y=\dim Z$ if $Z \neq \emptyset$.
\end{lem}

N. Das used the following lemma in \cite{Das}.

\begin{lem}[{\cite[Lemma 2.1]{Das}}]\label{bir1}
Any endomorphism which satisfies the conditions of Conjecture \ref{M} is birational.
\end{lem}

\begin{rmk}\label{bir2}
Actually, the proof of Lemma \ref{bir1} in \cite{Das} implies that any injective dominant morphism of varieties over $k$ is birational.
\end{rmk}

Due to Remark \ref{bir2}, in many situations, we can use a variant of Zariski's main theorem as follows.

\begin{thm}[{\cite[{\S}9, Original form]{Mum}}]\label{ZMT}
Let $f:V \to U$ be a birational morphism of varieties over $k$. If $f$ has finite fibers and $U$ is normal, then $f$ is an open embedding.
\end{thm}

Combining Remark \ref{bir2} and Theorem \ref{ZMT}, we have the following proposition.

\begin{prop}\label{bijiso}
Let $f:V \to U$ be a bijective morphism of varieties over $k$. If $U$ is normal, then $f$ is an isomorphism.
\end{prop}

\begin{proof}
By Remark \ref{bir2}, $f$ is birational. Thus, by Theorem \ref{ZMT}, $f$ is an open embedding. Since $f$ is bijective, the image of $f$ is $U$, hence, $f$ is an isomorphism.
\end{proof}

Since we may assume that $X$ is normal by Lemma \ref{normal}, it is enough to show that $\varphi$ is bijective by Proposition \ref{bijiso}.

We will use the notion of rational contractions introduced in \cite{HK}. C. Casagrande gives a useful characterization of rational contractions in \cite{Cas} as follows. 

\begin{prop}[{\cite[Remark 2.2]{Cas}}]\label{ratcon}
Let $f:X \dashrightarrow Y$ be a birational map, where X is a normal $\mathbb{Q}$-factorial projective variety over k, and Y is a normal projective variety over k. Since f is birational, we have nonempty open subsets U $\subset$ X and V $\subset$ Y such that f induces an isomorphism between U and V. We also take a resolution of f, that is, a smooth projective variety $\widetilde{X}$ over k with birational proper morphisms $p:\widetilde{X} \to X$ and $q:\widetilde{X} \to Y$ such that $q=f \circ p$. Then, the following are equivalent.

\begin{enumerate}
\item $f$ is a rational contraction,

\item $\codim (Y \setminus V) \ge 2$,

\item every $p$-exceptional prime divisor is also $q$-exceptional.
\end{enumerate}
\end{prop}

We will only use the equivalence of $(2)$ and $(3)$ in Proposition \ref{ratcon}, so we do not mention the definition of rational contractions.

\begin{rmk}\label{invratcon}
In the settings of Proposition \ref{ratcon}, we note that $\widetilde{X}$ also gives a resolution of $f^{-1}:Y \dashrightarrow X$. By the equivalence of $(2)$ and $(3)$ in Proposition \ref{ratcon}, if $\codim(X \setminus U) \ge 2$ and $\codim(Y \setminus V) \ge2$, then the set of $p$-exceptional prime divisors coincides with the set of $q$-exceptional prime divisors.
\end{rmk}

We will also use minimal model program. We mention the following proposition on the uniqueness of minimal models and canonical models. See \cite{KM} for minimal model program, minimal models, and canonical models.

\begin{prop}[{\cite[Theorem 3.52]{KM}}]\label{unique}
Let $V$ be a $\mathbb{Q}$-factorial normal projective variety.

\begin{enumerate}
\item Let $f:V \dashrightarrow M$ and $g:V \dashrightarrow M^{\prime}$ be two minimal models of $V$. Then, the set of exceptional prime divisors of f coincides with that of g.
\item Canonical models of $V$ are isomorphic to each other.
\end{enumerate}
\end{prop}

\section{Proof of Theorem \ref{main1}}
Let $\overline{X}$ be a $\mathbb{Q}$-factorial normal projective variety over $k$, $X$ a dense open subset of $\overline{X}$ with $\codim(\overline{X} \setminus X) \ge 2$, $Y$ a closed subset of $X$ with $\codim Y \ge 2$, $\varphi$ an endomorphism of $X$ such that $\varphi$ is injective on $X \setminus Y$. We assume that either the canonical divisor $K_{\overline{X}}$ of $\overline{X}$ or $-K_{\overline{X}}$ is ample. Our purpose is to prove that $\varphi$ is bijective.

By Lemma \ref{bir1}, $\varphi$ is birational, so we can regard $\varphi$ as a birational transformation of $\overline{X}$, $\varphi:\overline{X} \dashrightarrow \overline{X}$. In particular, by Proposition \ref{bijiso}, $\varphi$ induces an isomorphism between open sets of $\overline{X}$, $X \setminus Y$ and $X \setminus Z$, where $Z=X \setminus \varphi(X \setminus Y)$. Note that $\overline{X} \setminus (X \setminus Y)=(\overline{X} \setminus X) \cup \overline{Y}$, $\codim(\overline{X} \setminus X) \ge 2$, and $\codim \overline{Y}=\dim \overline{X}-\dim \overline{Y}=\dim X-\dim Y=\codim Y \ge 2$, so $\codim(\overline{X} \setminus (X \setminus Y)) \ge 2$, where $\overline{Y}$ is the closure of $Y$ in $\overline{X}$. By Lemma \ref{Z}, we can make the same calculation for $X \setminus Z$. Thus, we have $\codim(\overline{X} \setminus (X \setminus Z)) \ge 2$.

Let $\widetilde{X}$ be a resolution of $\varphi:\overline{X} \dashrightarrow \overline{X}$ with birational proper morphisms $p:\widetilde{X} \to \overline{X}$ and $q:\widetilde{X} \to \overline{X}$ such that $q=\varphi \circ p$. By pulling back the canonical divisor $K_{\overline{X}}$ of $\overline{X}$ by $p$, we have
$$K_{\widetilde{X}}=p^{\ast}K_{\overline{X}}+\sum_i a(E_i, \overline{X})E_i$$
in the N\'{e}ron-Severi group of $\widetilde{X}$, where $E_i$ runs over all $p$-exceptional prime divisors and $a(E_i, \overline{X})$ is a rational number, called the discrepancy of $E_i$ over $\overline{X}$. See \cite[Definition 2.22]{KM} for discrepancies and the above formula. By pulling back $K_{\widetilde{X}}$ by $q$, we also have
$$K_{\widetilde{X}}=q^{\ast}K_{\overline{X}}+\sum_j a(D_j, \overline{X})D_j$$
in the N\'{e}ron-Severi group of $\widetilde{X}$, where $D_j$ runs over all $q$-exceptional prime divisors. We have already seen that $\varphi$ induces an isomorphism between $X \setminus Y$ and $X \setminus Z$ and $\codim(\overline{X} \setminus (X \setminus Y)) \ge 2$, $\codim (\overline{X} \setminus (X \setminus Z)) \ge 2$. Thus, $\varphi$ satisfies the conditions of Remark \ref{ratcon}. Then we have that the set of $p$-exceptional prime divisors coincides with the set of $q$-exceptional prime divisors. Thus, the above formulas give $p^{\ast}K_{\overline{X}}=q^{\ast}K_{\overline{X}}$ in the N\'{e}ron-Severi group of $\widetilde{X}$. So we can identify $p^{\ast}K_{\overline{X}}$ with $q^{\ast}K_{\overline{X}}$ if we consider intersection numbers between these divisors and any curves.

If we take a curve $C$ in $Y$ such that $\varphi(C)$ is a point, we can take a curve $\widetilde{C}$ in $\widetilde{X}$ such that $p(\widetilde{C})=\overline{C}$. Since $\varphi(C)$ is a point and $q=\varphi \circ p$, $q(\widetilde{C})$ is also a point. Thus, we have
$$0=(K_{\overline{X}}, q_{\ast}\widetilde{C})=(q^{\ast}K_{\overline{X}}, \widetilde{C})=(p^{\ast}K_{\overline{X}}, \widetilde{C})=(K_{\overline{X}}, p_{\ast}\widetilde{C})=\deg(\widetilde{C} \slash \overline{C})(K_{\overline{X}}, \overline{C}).$$
Since $\deg(\widetilde{C} \slash \overline{C})$ is a positive integer, we have $(K_{\overline{X}}, \overline{C})=0$.

Now, since either $K_{\overline{X}}$ or $-K_{\overline{X}}$ is ample, a curve $\overline{C}$ with $(K_{\overline{X}}, \overline{C})=0$ cannot exist. Thus, $\varphi$ has finite fibers, so $\varphi$ is an open embedding by Theorem \ref{ZMT}. Hence, $\varphi$ is bijective by Theorem \ref{A}.

\section{Proof of Theorem \ref{main2}}
Let $\overline{X}$ be a $\mathbb{Q}$-factorial normal projective variety over $k$, $X$ a dense open subset of $\overline{X}$ with $\codim(\overline{X} \setminus X) \ge 2$, $Y$ a closed subset of $X$ with $\codim Y \ge 2$, $\varphi$ an endomorphism of $X$ such that $\varphi$ is injective on $X \setminus Y$. Suppose $\overline{X}$ has canonical singularities and $\overline{X}$ has the canonical model $f^{\can}:\overline{X} \dashrightarrow \overline{X}^{\can}$ which is obtained by minimal model program. We assume that each of the steps of minimal model program to obtain $\overline{X}^{\can}$ from $\overline{X}$ is a divisorial contraction as in Theorem \ref{main2}. We will prove that if we can take a curve $C$ in $Y$ such that $\varphi(C)$ is a point, there is a curve $\overline{C}^{\can}$ in $\overline{X}^{\can}$ with $(K_{\overline{X}^{\can}}, \overline{C}^{\can})=0$ by using minimal model program. The existence of such a curve $\overline{C}^{\can}$ contradicts the fact that $K_{\overline{X}^{\can}}$ is ample. Thus, we can prove that $\varphi$ is bijective as we have done in the proof of Theorem \ref{main1}.

We begin to run minimal model program. Let $f:\overline{X} \to \overline{X}^{\prime}$ be a divisorial contraction. We note that we can prove $(K_{\overline{X}}, \overline{C})=0$ in the same way as Section 2, where $\overline{C}$ is the closure of $C$ in $\overline{X}$. This implies that $\overline{C}$ is not contracted by $f$. Since $f(\overline{C})$ is not a point, we take a curve $\overline{C}^{\prime}=f(\overline{C})$ in $\overline{X}^{\prime}$. We will use the following proposition and lemma to prove $(K_{\overline{X}^{\prime}}, \overline{C}^{\prime})=0$.

\begin{prop}\label{excdivcan}
Let $V$ be a $\mathbb{Q}$-factorial normal projective variety which has the canonical model $V^{\can}$. If we have two birational maps $f:V \dashrightarrow V^{\can}$ and $g:V \dashrightarrow V^{\can}$, then the set of exceptional prime divisors of $f$ with positive discrepancies over $V^{\can}$ coincides with that of $g$.
\end{prop}

\begin{proof}
We imitate the proof of Proposition \ref{unique} $(1)$ in \cite{KM}. Take a proper birational morphism $h:U \to V$ which makes $f \circ h$ and $g \circ h$ morphisms, where $U$ is a $\mathbb{Q}$-factorial normal projective variety. Let $\Delta_U$ be a divisor over $U$ such that $K_U+\Delta_U=h^{\ast}K_V$ and $h_{\ast}\Delta_U=0$. Then, the proof of Proposition \ref{unique} in \cite{KM} tells us that there is an effective $(f \circ h)$-exceptional divisor $Z_1$ which satisfies $K_U+\Delta_U=(f \circ h)^{\ast}K_{V^{\can}}+Z_1$ in the N\'{e}ron-Severi group of $U$. Note that the set of prime divisors which appear in $h_{\ast}Z_1$ coincides with the set of $f$-exceptional prime divisors with a positive discrepancy over $V^{\can}$. Similarly, we also have $Z_2$ for $g \circ h$. In the N\'{e}ron-Severi group of $U$, we have $(f \circ h)^{\ast}K_{V^{\can}}-(g \circ h)^{\ast}K_{V^{\can}}=Z_2-Z_1$. Since $K_{V^{\can}}$ is ample, $(f \circ h)^{\ast}K_{V^{\can}}-(g \circ h)^{\ast}K_{V^{\can}}$ is $(g \circ h)$-nef, and $(g \circ h)^{\ast}K_{V^{\can}}-(f \circ h)^{\ast}K_{V^{\can}}$ is $(f \circ h)$-nef. By negativity lemma in \cite[Lemma 3.39]{KM}, we have $Z_2-Z_1$ and $Z_1-Z_2$ are effective. This implies $Z_1=Z_2$, and $h_{\ast}Z_1=h_{\ast}Z_2$, hence we have proved Proposition \ref{excdivcan}.
\end{proof}

\begin{lem}\label{iterate}
For any positive integer $m$, let $Y_m=Y \cup \varphi^{-1}(Y) \cup \cdots \cup (\varphi^{m-1})^{-1}(Y)$. Then, $\varphi^m$ is injective on $X \setminus Y_m$, and $\codim Y_m \ge 2$. In particular, we may replace $\varphi$ by $\varphi^m$ to prove Conjecture \ref{M}.
\end{lem}

\begin{proof}
Since $\varphi$ is injective on $X \setminus Y$ and $\codim Y \ge 2$, the claims are obvious.
\end{proof}

We have two birational maps $f^{\can}:\overline{X} \dashrightarrow \overline{X}^{\can}$ and $f^{\can} \circ \varphi:\overline{X} \dashrightarrow \overline{X}^{\can}$, so Proposition \ref{excdivcan} tells us that for all $f^{\can}$-exceptional prime divisor $E$ with positive discrepancy over $\overline{X}^{\can}$, the closure of $\varphi(E)$ in $\overline{X}$ is also a $f^{\can}$-exceptional prime divisor with positive discrepancy over $\overline{X}^{\can}$. Here, we used the fact that $\varphi$ does not contract any divisors of $X$ to show that the closure of $\varphi(E)$ is a prime divisor for any $E$. Since $\varphi$ is injective on $X \setminus Y$ and $\codim Y \ge 2$, we can identify $\varphi$ as a permutation of $f^{\can}$-exceptional prime divisors with positive discrepancies. The number of $f^{\can}$-exceptional prime divisors with positive discrepancies over $\overline{X}^{\can}$ is finite, so by using Lemma \ref{iterate} to replace $\varphi$ with some iteration $\varphi^m$, we can make $\varphi$ the trivial permutation of $f^{\can}$-exceptional prime divisors with positive discrepancies over $\overline{X}^{\can}$. Thus, if the discrepancy of $D$ over $\overline{X}^{\can}$ is positive, we may suppose $\varphi^{-1}(D) \subset D$. Now, we use the notion of log discrepancies. Let $\widetilde{X}$ be a resolution of $\varphi:\overline{X} \dashrightarrow \overline{X}$ with birational proper morphisms $p:\widetilde{X} \to \overline{X}$ and $q:\widetilde{X} \to \overline{X}$ such that $q=\varphi \circ p$. By pulling back a log canonical divisor $K_{\overline{X}}+D$ of $\overline{X}$ by $p$, we have
$$K_{\widetilde{X}}+(p^{-1})_{\ast}D=p^{\ast}(K_{\overline{X}}+D)+\sum_i a(E_i, \overline{X}, D)E_i$$
in the N\'{e}ron-Severi group of $\widetilde{X}$, where $E_i$ runs over all $p$-exceptional prime divisors and $a(E_i, \overline{X}, D)$ is a rational number, called the log discrepancy of $E_i$ with respect to $(\overline{X}, D)$. See \cite[Definition 2.25, Notation 2.26]{KM} for log discrepancies and the above formula. By pulling back $K_{\widetilde{X}}+D$ by $q$, we also have
$$K_{\widetilde{X}}+(q^{-1})_{\ast}D=q^{\ast}(K_{\overline{X}}+D)+\sum_j a(D_j, \overline{X}, D)D_j$$
in the N\'{e}ron-Severi group of $\widetilde{X}$, where $D_j$ runs over all $q$-exceptional prime divisors. We have already seen in Section $2$ that the set of $p$-exceptional prime divisors coincides with the set of $q$-exceptional prime divisors. Furthermore, since $q=\varphi \circ p$, we have $(p^{-1})_{\ast}D=(q^{-1})_{\ast}D$. Thus, the above formulas give $p^{\ast}(K_{\overline{X}}+D)=q^{\ast}(K_{\overline{X}}+D)$ in the N\'{e}ron-Severi group of $\widetilde{X}$. So we can identify $p^{\ast}(K_{\overline{X}}+D)$ with $q^{\ast}(K_{\overline{X}}+D)$ if we consider intersection numbers between these divisors and any curves. Let $\widetilde{C}$ be a curve in $\widetilde{X}$ such that $p(\widetilde{C})=\overline{C}$. Then, we have
\begin{align*}
0 &= (K_{\overline{X}}+D, q_{\ast}\widetilde{C})=(q^{\ast}(K_{\overline{X}}+D), \widetilde{C})=(p^{\ast}(K_{\overline{X}}+D), \widetilde{C})=(K_{\overline{X}}+D, p_{\ast}\widetilde{C})\\
&= \deg(\widetilde{C} \slash \overline{C})(K_{\overline{X}}+D, \overline{C}).
\end{align*}
Since $\deg(\widetilde{C} \slash \overline{C})$ is a positive integer, we have $(K_{\overline{X}}+D, \overline{C})=0$. We have already seen $(K_{\overline{X}}, \overline{C})=0$, therefore $(D, \overline{C})=0$. Hence, we have
\begin{align*}
0 &= (K_{\overline{X}}-a(D, \overline{X}^{\prime})D, \overline{C})=(f^{\ast}K_{\overline{X}^{\prime}}, \overline{C})=(K_{\overline{X}^{\prime}}, f_{\ast}\overline{C})\\
&= \deg (\overline{C} \slash f(\overline{C}))(K_{\overline{X}^{\prime}}, f(\overline{C})).
\end{align*}
Since $\deg (\overline{C} \slash f(\overline{C}))$ is a positive integer, we have $(K_{\overline{X}^{\prime}}, f(\overline{C}))=0$ as desired.

Also, in the case that the discrepancy of $D$ over $\overline{X}^{\can}$ is not positive, we can prove $(K_{\overline{X}^{\prime}}, f(\overline{C}))=0$. If $a(D, \overline{X}^{\can}) \le 0$, then we have
\begin{align*}
0 \ge a(D, \overline{X}^{\can}) \ge a(D, \overline{X})=0
\end{align*}
by the definition of canonical models, so we have $a(D, \overline{X}^{\can})=0$. Since $\overline{X}$ has canonical singularities, $\overline{X}^{\prime}$ also has canonical singularities. Thus, we have
\begin{align*}
0=a(D, \overline{X}^{\can}) \ge a(D, \overline{X}^{\prime}) \ge 0.
\end{align*}
This implies $a(D, \overline{X}^{\prime})=0$. Hence, we have
\begin{align*}
0=(K_{\overline{X}}, \overline{C})=(f^{\ast}K_{\overline{X}^{\prime}}, \overline{C})=(K_{\overline{X}^{\prime}}, f_{\ast}\overline{C})=\deg (\overline{C} \slash \overline{C}^{\prime})(K_{\overline{X}^{\prime}}, \overline{C}^{\prime}).
\end{align*}
Since $\deg (\overline{C} \slash \overline{C}^{\prime})$ is a positive integer, we have $(K_{\overline{X}^{\prime}}, \overline{C}^{\prime})=0$ as desired.

Note that $\varphi$ induces $\varphi^{\prime}:\overline{X}^{\prime} \dashrightarrow \overline{X}^{\prime}$ but $\varphi^{\prime}$ need not be defined on $\overline{C}^{\prime}$. This is because $\overline{C}$ may be contained in $D$. However, we only use the fact that $(K_{\overline{X}}, \overline{C})=0$ to define $\overline{C}^{\prime}$ and deduce $(K_{\overline{X}^{\prime}}, \overline{C}^{\prime})=0$. Thus, we can replace $\overline{X}$ to $\overline{X}^{\prime}$ and $\overline{C}$ to $\overline{C}^{\prime}$, and repeat the same operations until we obtain $\overline{X}^{\can}$. Finally, we have a curve $\overline{C}^{\can}$ in $\overline{X}^{\can}$ such that $(K_{\overline{X}^{\can}}, \overline{C}^{\can})=0$, and this is a contradiction.

\section{Proof of Theorem \ref{main3}}
Let $X$ be a dense open subset of a $\mathbb{Q}$-factorial normal projective threefold $\overline{X}$ over $k$ and $\codim(\overline{X} \setminus X) \ge 2$. Suppose $\overline{X}$ has canonical singularities and $\overline{X}$ has the canonical model $\overline{X}^{\can}$ which is obtained by minimal model program. Let $\overline{X}=\overline{X}_0 \dashrightarrow \overline{X}_1 \dashrightarrow \cdots \dashrightarrow \overline{X}_n=\overline{X}^{\can}$ be minimal model program to obtain $\overline{X}^{\can}$ from $\overline{X}$. Let $\varphi$ be an endomorphism of $X$ which satisfies the conditions of Conjecture \ref{M} but is not an automorphism. We can take a curve $C$ in $X$ such that $\varphi(C)$ is a point $P \in X$ as in Section $2$. We will prove that if either of two conditions in Theorem \ref{main3} holds for all $i$ such that $f_i$ is a flip, then $\varphi$ is an automorphism of $X$.

First, we will define birational transformations $\varphi_i$, curves $\overline{C}_i$, and points $P_i$ which appear in Theorem \ref{main3}. We define $\varphi_0=\varphi:\overline{X}_0 \dashrightarrow \overline{X}_0$, and for any $i=1, 2,\dots, n-1$, we inductively define a birational transformation $\varphi_i:\overline{X}_i \dashrightarrow \overline{X}_i$ such that $\varphi_i \circ f_{i-1}=f_{i-1} \circ \varphi_{i-1}$. We set $\overline{C}_0=\overline{C}$, which is the closure of $C$ in $\overline{X}$, and $P_0=\varphi(C)$.

If $f_0$ is a divisorial contraction, then we define $\overline{C}_1=f_0(\overline{C}_0)$, which is a curve because $(K_{\overline{X}_0}, \overline{C}_0)$ is zero as in Section 2 and this implies $\overline{C}_0$ is not contracted by $f_0$. We also define $P_1=f_0(P_0).$ In this case, we have $(K_{\overline{X}_1}, \overline{C}_1)=0$ as in Section 3.

If $f_0$ is a flip, then we use the following commutative diagram which is obtained by taking resolutions of several birational maps.

$$
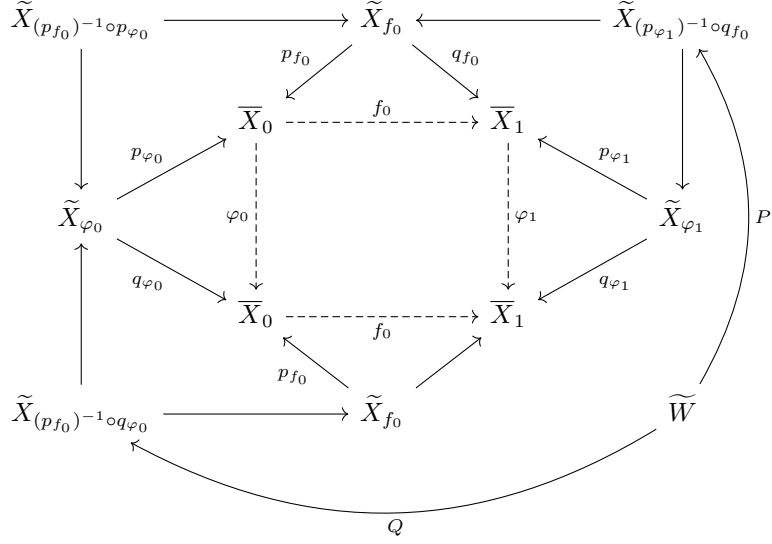
\begin{tikzcd}
\widetilde{X}_{(p_{f_0})^{-1} \circ p_{\varphi_0}} \arrow[rr] \arrow[dd] & & \widetilde{X}_{f_0} \arrow[ld, "p_{f_0}"'] \arrow[rd, "q_{f_0}"] & & \widetilde{X}_{(p_{\varphi_1})^{-1} \circ q_{f_0}} \arrow[ll] \arrow[dd]\\
 & \overline{X}_0 \arrow[rr, "f_0", dashed] \arrow[dd, "\varphi_0"', dashed] & & \overline{X}_1 \arrow[dd, "\varphi_1", dashed] & \\
\widetilde{X}_{\varphi_0} \arrow[ru, "p_{\varphi_0}"] \arrow[rd, "q_{\varphi_0}"'] & & & & \widetilde{X}_{\varphi_1} \arrow[lu, "p_{\varphi_1}"'] \arrow[ld, "q_{\varphi_1}"] \\
 & \overline{X}_0 \arrow[rr, "f_0"', dashed] & & \overline{X}_1 & \\
\widetilde{X}_{(p_{f_0})^{-1} \circ q_{\varphi_0}} \arrow[rr] \arrow[uu] & & \widetilde{X}_{f_0} \arrow[lu, "p_{f_0}"] \arrow[ru] & & \widetilde{W} \arrow[llll, "Q", bend left] \arrow[uuuu, "P"', bend right]
\end{tikzcd}$$

In the above diagram, $\widetilde{X}_{\varphi_0}$ means a resolution of $\varphi_0$ with proper birational morphisms $p_{\varphi_0}$ and $q_{\varphi_0}$, where $p_{\varphi_0}$ is a morphism from $\widetilde{X}_{\varphi_0}$ to the domain of $\varphi_0$ and $q_{\varphi_0}$ is a morphism from $\widetilde{X}_{\varphi_0}$ to the codomain of $\varphi_0$. We use the same notations for other birational maps except for $\widetilde{W}$, which is a resolution of the composition of the birational maps $\widetilde{X}_{(p_{\varphi_1})^{-1} \circ q_{f_0}} \to \widetilde{X}_{f_0} \dashrightarrow \widetilde{X}_{(p_{f_0})^{-1} \circ p_{\varphi_0}} \to \widetilde{X}_{\varphi_0} \dashrightarrow \widetilde{X}_{(p_{f_0})^{-1} \circ q_{\varphi_0}}$ with proper birational morphisms $P$ and $Q$, where $P$ is a morphism from $\widetilde{W}$ to $\widetilde{X}_{(p_{\varphi_1})^{-1} \circ q_{f_0}}$ and $Q$ is a morphism from $\widetilde{W}$ to $ \widetilde{X}_{(p_{f_0})^{-1} \circ q_{\varphi_0}}$.

We will take curves in varieties in the diagram in turn. Take $C_1$ in $\widetilde{X}_{\varphi_0}$ such that $p_{\varphi_0}(C_1)=\overline{C}_0$, $C_2$ in $\widetilde{X}_{(p_{f_0})^{-1} \circ q_{\varphi_0}}$ such that $C_2$ maps onto $C_1$ by the morphism $\widetilde{X}_{(p_{f_0})^{-1} \circ q_{\varphi_0}} \to \widetilde{X}_{\varphi_0}$, and $C_3$ in $\widetilde{W}$ such that $Q(C_3)=C_2$. The choices of these curves may not be unique, but it does not matter. Next, let $C_4$ in $\widetilde{X}_{(p_{\varphi_1})^{-1} \circ q_{f_0}}$ be $P(C_3)$, let $C_5$ in $\widetilde{X}_{f_0}$ be the image of $C_4$ by the morphism $\widetilde{X}_{(p_{\varphi_1})^{-1} \circ q_{f_0}} \to \widetilde{X}_{f_0}$, let $C_6$ in $\widetilde{X}_{\varphi_1}$ be the image of $C_4$ by the morphism $\widetilde{X}_{(p_{\varphi_1})^{-1} \circ q_{f_0}} \to \widetilde{X}_{\varphi_1}$, and let $C_7$ in $\overline{X}_1$ be $q_{f_0}(C_5)$. We note that these are actually curves, not points, by the commutativity of the diagram and the fact that $f_0$ is defined on some nonempty open subsets of $\overline{C}_0$. Actually, $(K_{\overline{X}_0}, \overline{C}_0)=0$ and $\dim \overline{X}_0=3$ imply that $\overline{C}_0$ is not contained in the flipping locus. We also note that $C_7$ is equal to $p_{\varphi_1}(C_6)$ by the commutativity of the diagram. We define $\overline{C}_1=C_7$.

If $f_0$ satisfies the first condition in Theorem \ref{main3}, that is, $P_0$ is defined and $f_0$ is defined at $P_0$, then $q_{\varphi_1}(C_6)$ is a point of $\overline{X}_1$ by the commutativity of the diagram. We define $P_1=q_{\varphi_1}(C_6)$. In this case, we deduce $(K_{\overline{X}_1}, \overline{C}_1)=0$ as we have $(K_{\overline{X}_0}, \overline{C}_0)=0$.

If $f_0$ satisfies the second condition in Theorem \ref{main3}, that is, there exists an open subset $U$ of $\overline{X}_i$ such that $f_0$ induces an isomorphism between $U$ and $f_0(U)$, and $U$ contains $\overline{C}_i$, then we cannot define $P_1$. However, we can deduce $(K_{\overline{X}_1}, \overline{C}_1)=0$ as follows. By pulling back $K_{\overline{X}_0}$ by $p_{f_0}$, we have
$$K_{\widetilde{X}_{f_0}}=(p_{f_0})^{\ast}K_{\overline{X}_0}+\sum_j a(D_j, \overline{X}_0)D_j$$
in the N\'{e}ron-Severi group of $\widetilde{X}_{f_0}$, where $D_j$ runs over all $p_{f_0}$-exceptional prime divisors. By pulling back $K_{\overline{X}_1}$ by $q_{f_0}$, we have
$$K_{\widetilde{X}_{f_0}}=(q_{f_0})^{\ast}K_{\overline{X}_1}+\sum_k a(E_k, \overline{X}_1)E_k$$
in the N\'{e}ron-Severi group of $\widetilde{X}_{f_0}$, where $E_k$ runs over all $q_{f_0}$-exceptional prime divisors. By Remark \ref{invratcon}, we have the set of $D_j$ coincides with the set of $E_k$. Note that for some $D_j$, a discrepancy $a(D_j, \overline{X}_0)$ may differ from $a(D_j, \overline{X}_1)$. However, there is a fact that such a $D_j$ maps into the flipping locus by $p_{f_0}$. See \cite[Proposition 5-1-11]{KMM} for this fact. Thus, such a $D_j$ does not intersect with $C_5$ by the assumption and we have $(D_j, C_5)=0$. Hence, we have
\begin{align*}
\deg(C_5 \slash \overline{C}_1)(K_{\overline{X}_1}, \overline{C}_1) &=(K_{\overline{X}_1}, (q_{f_0})_{\ast}C_5)=((q_{f_0})^{\ast}K_{\overline{X}_1}, C_5)\\
&=(K_{\widetilde{X}_{f_0},} C_5)-\sum_k a(E_k, \overline{X}_1)(E_k, C_5)\\
&=(K_{\widetilde{X}_{f_0}}, C_5)-\sum_j a(D_j, \overline{X}_0)(D_j, C_5)\\
&=((p_{f_0})^{\ast}K_{\overline{X}_0}, C_5)\\
&=(K_{\overline{X}_0}, (p_{f_0})_{\ast}C_5)=\deg(C_5 \slash \overline{C}_0)(K_{\overline{X}_0}, \overline{C}_0)=0.
\end{align*}
Since $\deg(C_5 \slash \overline{C}_1)$ is a positive integer, we have $(K_{\overline{X}_1}, \overline{C}_1)=0$ as desired.

To summarize three cases, we define a curve $\overline{C}_1$ in $\overline{X}_1$ with $(K_{\overline{X}_1}, \overline{C}_1)=0$ by using a curve $\overline{C}_0$ in $\overline{X}_0$ with $(K_{\overline{X}_0}, \overline{C}_0)=0$. Moreover, if there is a curve $C^{\prime}$ in a resolution $\widetilde{X}_{\varphi_1}$ of $\varphi_1$ such that $p_{\varphi_1}(C^{\prime})=\overline{C}_1$ and $q_{\varphi_1}(C^{\prime})$ is a point of $\overline{X}_1$, then we define a point $P_1$ in $\overline{X}_1$ as $q_{\varphi_1}(C^{\prime})$. Note that even if $f_0$ is a divisorial contraction, $f_0$ fits into the above diagram, where some resolutions are not needed. We can repeat this operation again and again unless a flip which does not satisfy both of the conditions in Theorem \ref{main3} appears. Finally, we have a curve $\overline{C}_n$ in $\overline{X}_n=\overline{X}^{\can}$ such that $(K_{\overline{X}_n}, \overline{C}_n)=0$, and this contradicts the fact that $K_{\overline{X}_n}=K_{\overline{X}^{\can}}$ is ample.

\begin{rmk}
Note that there may be a flip $f_i$ satisfying both of the conditions in Theorem \ref{main3}. For such a flip, we may ignore the second condition in Theorem \ref{main3} and we can define $P_{i+1}$. Also note that we can define $P_i$ inductively until a flip which is not defined at $P_{i-1}$ appears. 
\end{rmk}

\section{Observations on some other cases}
Let $X$ be a dense open subset of a $\mathbb{Q}$-factorial normal projective variety $\overline{X}$ over $k$ and $\codim(\overline{X} \setminus X) \ge 2$. In this section, we observe the cases where we can obtain a minimal model or a Mori fiber space by minimal model program on $\overline{X}$.

First, we consider the case that $\overline{X}$ has a minimal model $f^{\min}:\overline{X} \dashrightarrow \overline{X}^{\min}$ obtained by minimal model program. Let $f:\overline{X} \dashrightarrow \overline{X}^{\prime}$ is the first step of minimal model program. If $f$ is a divisorial contraction, then we can do the same operation as in Section $3$. In fact, both of $f^{\min}$ and $f^{\min} \circ \varphi$ give a minimal model of $\overline{X}$. Proposition \ref{unique}$(1)$ and the fact that $\varphi$ does not contract any divisors tell us that $\varphi$ can be identified as a permutation of $f^{\min}$-exceptional prime divisors. The number of $f^{\min}$-exceptional prime divisors is finite, so by using Lemma \ref{iterate} to replace $\varphi$ with some iteration $\varphi^m$, we can make $\varphi$ the trivial permutation of $f^{\min}$-exceptional prime divisors. Thus, we can obtain a curve $\overline{C}^{\prime}=f(\overline{C})$ in $\overline{X}^{\prime}$ with $(K_{\overline{X}^{\prime}}, \overline{C}^{\prime})=0$ as in Section $3$. If $f$ is a flip and $\dim \overline{X}=3$, we also can do the same operation as in Section $4$ provided either of the conditions in Theorem \ref{main3} holds. After repeating the operations, we obtain $\overline{X}^{\min}$ and a curve $\overline{C}^{\min}$ in $\overline{X}^{\min}$ with $(K_{\overline{X}^{\min}}, \overline{C}^{\min})=0$. However, $K_{\overline{X}^{\min}}$ is only nef, not ample, so this does not give a contradiction. Moreover, even if Conjecture \ref{M} is true for $\overline{X}^{\min}$, $\varphi$ is not necessarily bijective. This is because $\varphi^{\min}$, a birational transformation of $\overline{X}^{\min}$ induced by $\varphi$, may not be defined on $\overline{C}^{\min}$. Actually, if the following problem is solved for $\overline{X}^{\min}$, then we can deduce $\varphi$ is bijective.

\begin{prob}\label{problem}
Let $V$ be a $\mathbb{Q}$-factorial normal projective variety with the nef canonical divisor $K_{V}$. Let $\varphi$ be a birational transformation which induces an isomorphism between nonempty open subsets $U$, $U^{\prime}$ of $V$. We suppose that $\codim (V \setminus U)$ and $\codim (V \setminus U^{\prime})$ are at least $2$. Then, is $\varphi$ an automorphism of $V$?
\end{prob}

Note that Problem \ref{problem} is true for varieties with the ample canonical divisors as follows.

\begin{prop}\label{canbiriso}
Let $V$ be a $\mathbb{Q}$-factorial normal projective variety with the ample canonical divisor $K_{V}$. Let $\varphi$ be a birational transformation which induces an isomorphism between open subsets $U$, $U^{\prime}$ of $V$. We suppose that $\codim (V \setminus U)$ and $\codim (V \setminus U^{\prime})$ are at least $2$. Then, $\varphi$ is an automorphism of $V$.
\end{prop}

\begin{proof}
We consider the complete linear systems $\lvert K_V \rvert$ of $K_V$, and $\lvert \varphi^{\ast}K_V \rvert$ of $\varphi^{\ast}K_V$. Since $\codim(V \setminus U)\ge 2$ and $\codim(V \setminus U^{\prime}) \ge 2$, we have $\varphi^{\ast}K_V=K_V$. By assumption, $K_V$ is ample, so $\lvert K_V \rvert$ induces an embedding $i:V \to \mathbb{P}^l$, for some $l$. Also, $\lvert \varphi^{\ast}K_V \rvert$ induces an embedding $j:V \to \mathbb{P}^l$ such that $j=i \circ \varphi$. In generally, morphisms to $\mathbb{P}^l$ induced by the same complete linear system differ by automorphisms of $\mathbb{P}^l$ each other. Hence, $j=i \circ \varphi$ implies that $\varphi$ is an automorphism of $V$.
\end{proof}

Second, we consider the case that $\overline{X}$ is birational to a Mori fiber space $\psi:S \to T$. In this case, we can do the same operations as in the previous sections only for flips provided either of the conditions in Theorem \ref{main3} holds since there is no results on $f$-exceptional prime divisors as Proposition \ref{unique} or \ref{excdivcan}, where $f$ is a biratonal map $\overline{X} \dashrightarrow S$ obtained by minimal model program. This gives the following theorem generalizing the Fano-case of Theorem \ref{main1}, but it seems very specific.

\begin{thm}
Let $X$ be a dense open subset of a $\mathbb{Q}$-factorial normal projective threefold $\overline{X}$ over $k$ and $\codim(\overline{X} \setminus X) \ge 2$. Suppose $\overline{X}$ has a Fano model $F$ which is obtained by minimal model program. Let $\overline{X}=\overline{X}_0 \dashrightarrow \overline{X}_1 \dashrightarrow \cdots \dashrightarrow \overline{X}_n=F$ be a sequence given by minimal model program to obtain $F$ from $\overline{X}$. Let $\varphi$ be an endomorphism of $X$ which satisfies the conditions of Conjecture \ref{M} but is not an automorphism. Then there is a curve $C$ in $X$ such that $\varphi(C)$ is a point $P \in X$. In the way described in Section $4$, $\varphi_0=\varphi$ induces $\varphi_i:\overline{X}_i \dashrightarrow \overline{X}_i$ for all $i$, and $\overline{C}_0=\overline{C}$, which is the closure of $C$ in $\overline{X}$, transforms into a curve $\overline{C}_i$ in $\overline{X}_i$ for all $i$, and $P_0=P$ transforms into a point $P_i \in \overline{X}_i$ for some $i$. Then, there exists some $i$ such that $f_i:\overline{X}_i \dashrightarrow \overline{X}_{i+1}$ is a divisorial contraction, or $f_i$ is a flip and neither of the following conditions holds.

\begin{enumerate}
\item $P_i$ is defined and $f_i$ is defined at $P_i$.\\

\item There exists an open subset $U$ of $\overline{X}_i$ such that $f_i$ induces an isomorphism between $U$ and $f_i(U)$, and $U$ contains $\overline{C}_i$.
\end{enumerate}
\end{thm}

\begin{proof}
This theorem follows from the argument in Section $4$.
\end{proof}

\section*{Acknowledgements}
The author thanks Seidai Yasuda, my supervisor, for his invaluable guidance. The author also thanks Hisato Matsukawa for his advice on the content of this paper. This work was supported by JST SPRING, Grant Number JPMJSP2119.

\end{document}